\documentclass[11pt,a4paper,reqno]{amsart}

\usepackage[margin=1.05in]{geometry}
\usepackage[T1]{fontenc}
\usepackage[utf8]{inputenc}
\usepackage{lmodern}
\usepackage{amsmath,amssymb,amsthm,mathtools,mathrsfs}
\usepackage{enumitem}
\usepackage{xcolor}
\usepackage[colorlinks=true,linkcolor=blue!60!black,citecolor=red!60!black,urlcolor=blue!70!black]{hyperref}

\numberwithin{equation}{section}

\newtheorem{theorem}{Theorem}[section]
\newtheorem{lemma}[theorem]{Lemma}
\newtheorem{proposition}[theorem]{Proposition}
\newtheorem{corollary}[theorem]{Corollary}

\newtheorem{question}[theorem]{Question}
\theoremstyle{definition}

\newtheorem{example}[theorem]{Example}
\newtheorem{remark}[theorem]{Remark}

\newcommand{\R}{\mathbb{R}}
\newcommand{\B}{\mathbb{B}}                     
\newcommand{\Bcl}{\overline{\mathbb{B}}}        
\newcommand{\D}{\mathbb{D}}
\newcommand{\Sph}{\mathbb{S}}
\newcommand{\cA}{\mathcal{A}}
\newcommand{\cQ}{\mathcal{Q}}
\newcommand{\Hess}{\operatorname{Hess}}
\newcommand{\tr}{\operatorname{tr}}
\newcommand{\Int}{\operatorname{Int}}
\newcommand{\Crit}{\operatorname{Crit}}
\newcommand{\Ric}{\operatorname{Ric}}

\newcommand{\bn}{\mathbf n}

\title[Radial pinching for Gaussian $f$-minimal submanifolds]{Radial pinching and topological rigidity for free boundary Gaussian $f$-minimal submanifolds}
\author{Niang Chen}
\address{Faculty of Arts and Sciences, Beijing Normal University, Zhuhai 519087, China}
\email{chenniang@bnu.edu.cn}

\begin{document}

\begin{abstract}
Let $M^k\subset\Bcl_R^N$ be a smooth compact connected orientable free boundary
$f_c$-minimal submanifold of the closed Euclidean ball, where
$f_c(x)=c|x|^2/2$ and $c\ge0$.  Assume that
$cR^2\le k$ and
\[
 |A_{x^\perp}|^2\le
 1+\frac{1}{k-1}\bigl(1-c|x^\perp|^2\bigr)^2,
 \qquad
 A_{x^\perp}(X,Y)=\langle x^\perp,A(X,Y)\rangle.
\]
We prove that $M$ is diffeomorphic either to $\D^k$ or to
$\Sph^1\times\D^{k-1}$; strict pinching yields the disk.  The proof uses
Hessian convexity of the squared-distance function, a nullity estimate along
its minimum set, and a sublevel-set argument.  In dimension two and codimension
one, the non-disk branch is rotationally symmetric.  We also construct a local
family of embedded rotational examples for small $c\ge0$, with the $c=0$
member equal to the critical catenoid.
\end{abstract}

\maketitle

\section{Introduction}

Free boundary minimal submanifolds in Euclidean balls are related to the
Steklov eigenvalue problem, sharp eigenvalue bounds, and area estimates; see
\cite{FraserSchoen2011,FraserSchoen2016,Brendle2012}.  Rigidity results for the
equatorial disk and the critical catenoid include those of Nitsche,
Fraser--Schoen, and McGrath
\cite{Nitsche1985,FraserSchoen2015,McGrath2018}.

Ambrozio and Nunes proved that a compact free boundary minimal surface in
$\Bcl^3$ satisfying
\[
        |A|^2\langle x,\nu\rangle^2\le 2,
\]
where $\nu$ denotes a unit normal, is either the equatorial disk or the critical catenoid
\cite{AmbrozioNunes2021}.  Their proof combines convexity of the
squared-distance function with a Jacobi-field and nodal-set argument.  Li and
Xiong obtained corresponding results in geodesic balls of hyperbolic and
spherical space \cite{LiXiong2018}, and Barbosa and Viana extended the
Euclidean two-dimensional theorem to higher codimension
\cite{BarbosaViana2020}.  Related gap and uniqueness results in conformally
Euclidean balls and rotational domains appear in
\cite{BarbosaGoncalvesPereira2021,BarbosaFreitasMeloVitorio2023}.

For higher-dimensional free boundary minimal submanifolds, Barbosa and Viana
proved a topological rigidity under the full-norm condition
\[
        |x^\perp|^2|A|^2\le \frac{k}{k-1}:
\]
the submanifold is diffeomorphic either to a disk or to a solid tube
\cite{BarbosaViana2022}.  The argument below uses the same
Hessian-convexity and minimum-set framework, but treats the Gaussian
$f$-minimal equation and imposes a condition on the scalar radial form
$A_{x^\perp}$ rather than on the full normal-bundle norm of $A$.

Full-curvature methods give complementary cohomology-vanishing and index
estimates in both unweighted and weighted settings
\cite{ChenGe2022,ZhaoCao2024,ChenGeZhang2023,Chen2026}.  Those results use
global harmonic-form arguments and control the full second fundamental form,
whereas the present proof is based on the radial scalar form and the Hessian of
the squared-distance function.  At $c=0$, the Barbosa--Viana full-norm
condition implies the present radial condition.  For $c>0$, the two conditions
are not pointwise comparable in general because the radial threshold depends
on $c|x^\perp|^2$.

The pinching assumption is also relevant to the annular conclusion.
Fern\'an\-dez, Haus\-wirth, and Mira constructed infinitely many non-rotational
immersed free boundary minimal annuli in the unit ball
\cite{FernandezHauswirthMira2023}.  Thus the unrestricted immersed annulus
class contains examples other than the critical catenoid.  In dimension two
and codimension one, the radial pinching condition used here implies rotational
symmetry of the non-disk branch.

The Gaussian equation is related to the self-shrinker equation.  With the
normalization $f(x)=|x|^2/4$, corresponding to $c=1/2$, it becomes
$\vec H=-x^\perp/2$.  Curvature gap results for self-shrinkers include those of
Le--\v{S}e\v{s}um, Cao--Li, and Ding--Xin
\cite{LeSesum2011,CaoLi2013,DingXin2014}.  Those results concern complete or
closed self-shrinkers rather than free boundary submanifolds in a ball.

The squared-distance Hessian depends only on the scalar projection
$A_{x^\perp}$.  The condition $cR^2\le k$ and the radial pinching inequality
imply its nonnegativity; the proof is given in
Proposition~\ref{prop:pinch-hess}.  Sections~\ref{sec:pinching}--\ref{sec:topology}
analyze the minimum set and derive the topological alternatives.  Section~\ref{sec:surface-rigidity}
proves rotational symmetry in dimension two, and Section~\ref{sec:ode}
studies the rotational shooting problem near the critical catenoid.

\subsection{Notation and main results}

For $c\ge0$, set
\[
        f_c(x)=\frac c2|x|^2.
\]
We use the trace convention
\[
        \vec H=\tr A=\sum_{i=1}^k A(e_i,e_i).
\]
Thus an $f_c$-minimal immersion satisfies
\[
        \vec H+(\bar\nabla f_c)^\perp=0,
        \qquad\text{equivalently}\qquad
        \vec H=-cx^\perp.
\]
For tangent vectors $X,Y$, define the scalar radial second fundamental form by
\[
        A_{x^\perp}(X,Y)=\langle x^\perp,A(X,Y)\rangle.
\]
The notation $M\cong M'$ means that the abstract smooth manifolds $M$ and $M'$
are diffeomorphic.  When embedded submanifolds are said to be congruent up to
ambient rotations, their images are understood to be carried to one another by
an element of $SO(N)$; no classification of multiply covered immersions is
intended.

\begin{theorem}[Topological rigidity]\label{thm:mainA}
Let $k\ge2$, and let $M^k\subset\Bcl_R^N$ be a smooth compact connected
orientable free boundary $f_c$-minimal submanifold with nonempty boundary.
Assume that
\[
        cR^2\le k
\]
and
\begin{equation}\label{eq:main-pinching}
        |A_{x^\perp}|^2\le
        1+\frac{1}{k-1}\left(1-c|x^\perp|^2\right)^2
\end{equation}
on $M$.  Then
\[
        M\cong\D^k
        \qquad\text{or}\qquad
        M\cong\Sph^1\times\D^{k-1}.
\]
If \eqref{eq:main-pinching} is strict at every point of $M$, then
$M\cong\D^k$.
\end{theorem}

The Hessian, minimum-set, and equality statements used in the proof are stated
separately in Sections~\ref{sec:pinching}--\ref{sec:topology}.

\begin{theorem}[Rotational rigidity in the non-disk branch]\label{thm:mainB}
Let $\Sigma^2\subset\Bcl_R^3$ be a smooth compact connected orientable free
boundary $f_c$-minimal surface satisfying the assumptions of
Theorem~\ref{thm:mainA}.  If $\Sigma$ is not diffeomorphic to $\D^2$, then $\Sigma$ is invariant under a one-parameter group of
rotations about an axis through the origin.  In particular,
$\Sigma\cong\Sph^1\times\D^1$.
\end{theorem}

The conclusion of Theorem~\ref{thm:mainB} concerns invariance of the immersed
image.  The model-uniqueness statement in
Proposition~\ref{prop:conditional-uniqueness} is restricted to embedded
annuli, thereby excluding nontrivial angular multiple covers.

The following result concerns the rotational shooting problem near $c=0$.

\begin{theorem}[Perturbative Gaussian annular family]\label{thm:local-gaussian-family}
Fix $R>0$.  There exist $\varepsilon_R>0$ and a smooth Gaussian family,
including the unweighted endpoint $c=0$,
$\{\Sigma_c\}_{0\le c<\varepsilon_R}$ of embedded rotational free boundary
$f_c$-minimal annuli in $\Bcl_R^3$, where
\[
        f_c(x)=\frac c2|x|^2,
\]
such that the following statements hold.
\begin{enumerate}[label=\textup{(\roman*)},leftmargin=2.4em]
\item $\Sigma_0$ is the unweighted critical catenoid in $\Bcl_R^3$.
\item If $0<c<\varepsilon_R$, then $\Sigma_c$ is non-flat.
\item Every $\Sigma_c$ satisfies
\begin{equation}\label{eq:local-family-pinching}
        |A_{x^\perp}|^2
        \le 1+\bigl(1-c|x^\perp|^2\bigr)^2,
\end{equation}
and equality holds exactly on the minimum circle of $\Sigma_c$.
\item The corresponding neck radius is the unique shooting root in a
      neighborhood of the critical-catenoid neck radius.  No uniqueness
      assertion is made for roots elsewhere in the full shooting domain.
\end{enumerate}
\end{theorem}

Section~\ref{sec:ode} proves Theorem~\ref{thm:local-gaussian-family} by an
implicit-function argument at the critical catenoid and a factorization of the
pinching deficit.  The theorem gives local existence and local uniqueness near
the critical-catenoid root.  Global uniqueness of the shooting root is not
proved; Proposition~\ref{prop:conditional-uniqueness} records the corresponding
conditional statement.

\section{Preliminaries}

Throughout, submanifolds may be immersed unless embeddedness is explicitly stated.  The sets $C(M)$, geodesics, disk-bundle neighborhoods, and all topological conclusions are understood on the abstract manifold $M$; if the immersion has self-intersections, no assertion is made about the topology of the image as a subset of $\R^N$.

Let $M^k\subset \R^N$ be a smooth immersed submanifold.  We denote by $\bar\nabla$ the Euclidean connection and by $\nabla$ the Levi-Civita connection of $M$.  The second fundamental form is
\[
        A(X,Y)=\bar\nabla_XY-\nabla_XY,
\]
and the mean curvature vector is
\[
        \vec H=\sum_{i=1}^k A(e_i,e_i).
\]
The normal and tangential projections of an ambient vector $v$ are denoted by $v^\perp$ and $v^T$.

Throughout, $\B_R^N$ denotes the \emph{open} ball of radius $R$ centered at the origin in $\R^N$, $\Bcl_R^N$ its closure, and $\partial\B_R^N$ the boundary sphere.  A compact submanifold $M^k\subset \Bcl_R^N$ with nonempty boundary is a free boundary submanifold if
\begin{equation}\label{eq:freebdry-contact}
        M\cap \partial \B_R^N=\partial M,
\end{equation}
and $M$ meets $\partial \B_R^N$ orthogonally along $\partial M$.  Condition \eqref{eq:freebdry-contact} says that $M$ touches the sphere exactly along its boundary; equivalently, $\partial M\subset\partial\B_R^N$ and
\[
        \Int(M)\subset \B_R^N,
\]
that is, the interior of $M$ lies in the open ball.  If $\eta$ denotes the outward unit conormal of $\partial M\subset M$, then
\begin{equation}\label{eq:freebdry-basic}
        \eta=\frac{x}{R},
        \qquad
        x^T=R\eta
        \quad\text{on }\partial M.
\end{equation}
In particular, $x$ is tangent to $M$ along $\partial M$, hence
\begin{equation}\label{eq:xperp-boundary}
        x^\perp=0\qquad\text{on }\partial M.
\end{equation}
For the Gaussian potential
\[
        f(x)=\frac c2|x|^2,
\]
the $f$-minimal equation is
\[
        \vec H+(\bar\nabla f)^\perp=0.
\]
Since $\bar\nabla f=cx$, this gives
\begin{equation}\label{eq:fmin-basic}
        \vec H=-cx^\perp.
\end{equation}

We also use the drift Laplacian on functions
\begin{equation}\label{eq:drift-lap}
        \Delta_f u=\Delta u-\langle \nabla f,\nabla u\rangle.
\end{equation}
This is the sign convention used in the Jacobi equation in Section~\ref{sec:surface-rigidity}.  It should not be confused with the nonnegative weighted Hodge Laplacian sometimes denoted by the same symbol.

\section{Algebraic pinching and Hessian convexity}\label{sec:pinching}

We begin with the elementary Hessian identity for the squared-distance function.

\begin{lemma}[Hessian identity]\label{lem:hess-identity}
Let $r=|x|^2/2$.  For every $X,Y\in T_pM$,
\begin{equation}\label{eq:hess-r-general}
        \Hess_M r(X,Y)=\langle X,Y\rangle+\langle x^\perp,A(X,Y)\rangle.
\end{equation}
\end{lemma}

\begin{proof}
Since $\nabla^M r=x^T$, we compute
\[
\Hess_M r(X,Y)=\langle \nabla_X^M x^T,Y\rangle.
\]
Writing $x^T=x-x^\perp$ gives
\[
\langle \nabla_X^M x^T,Y\rangle
=\langle \bar\nabla_Xx,Y\rangle-\langle \bar\nabla_Xx^\perp,Y\rangle
=\langle X,Y\rangle+\langle x^\perp,\bar\nabla_XY\rangle,
\]
which is \eqref{eq:hess-r-general}.
\end{proof}

For the radial normal direction, define the scalar symmetric bilinear form
\begin{equation}\label{eq:Axperp-def}
        A_{x^\perp}(X,Y):=\langle x^\perp,A(X,Y)\rangle,
        \qquad
        |A_{x^\perp}|^2=\sum_{i,j=1}^k \langle x^\perp,A(e_i,e_j)\rangle^2.
\end{equation}
Only this scalar projection enters \eqref{eq:hess-r-general}, which may be written as
\begin{equation}\label{eq:hess-g-Axperp}
        \Hess_M r=g+A_{x^\perp}.
\end{equation}

For a normal vector field $\xi$, we write $\xi\otimes g$ for the normal-vector-valued symmetric two-tensor
\[
        (\xi\otimes g)(X,Y)=\xi\,g(X,Y).
\]
Let
\begin{equation}\label{eq:Phi-def}
        \Phi=A-\frac1k\,\vec H\otimes g,
        \qquad
        \Phi(X,Y)=A(X,Y)-\frac1k\vec H\,g(X,Y),
\end{equation}
be the traceless second fundamental form.  Using \eqref{eq:fmin-basic}, we obtain
\[
        A=\Phi-\frac c k\,x^\perp\otimes g,
        \qquad
        A(X,Y)=\Phi(X,Y)-\frac c k x^\perp g(X,Y).
\]
Substituting this into Lemma~\ref{lem:hess-identity} gives the Gaussian Hessian decomposition
\begin{equation}\label{eq:gaussian-hess-decomp}
        \Hess_M r(X,Y)
        =\left(1-\frac c k |x^\perp|^2\right)g(X,Y)
        +\langle x^\perp,\Phi(X,Y)\rangle.
\end{equation}

We use the following standard algebraic estimate.

\begin{lemma}[Traceless eigenvalue estimate]\label{lem:traceless-est}
Let $B$ be a traceless symmetric bilinear form on a $k$-dimensional Euclidean vector space.  Then
\begin{equation}\label{eq:traceless-est}
        \lambda_{\min}(B)\ge -\sqrt{\frac{k-1}{k}}\,|B|.
\end{equation}
\end{lemma}

\begin{proof}
Let $\lambda_1\le\cdots\le\lambda_k$ be the eigenvalues of $B$.  Since $\sum_i\lambda_i=0$,
\[
        -\lambda_1=\sum_{i=2}^k\lambda_i.
\]
By Cauchy--Schwarz,
\[
        \lambda_1^2\le (k-1)\sum_{i=2}^k\lambda_i^2
        =(k-1)(|B|^2-\lambda_1^2).
\]
Thus $k\lambda_1^2\le (k-1)|B|^2$, which gives \eqref{eq:traceless-est}.
\end{proof}

\begin{proposition}[Radial projection pinching implies Hessian convexity]\label{prop:pinch-hess}
Assume $cR^2\le k$.  If
\begin{equation}\label{eq:pinch-prop}
        |A_{x^\perp}|^2\le
        1+\frac{1}{k-1}\left(1-c|x^\perp|^2\right)^2
\end{equation}
on $M$, then
\[
        \Hess_M r\ge0.
\]
\end{proposition}

\begin{proof}
Fix a point of $M$ and put
\[
        s=|x^\perp|^2,
        \qquad
        \zeta=cs,
        \qquad
        B=A_{x^\perp}.
\]
By $f$-minimality,
\[
        \tr B=\langle x^\perp,\vec H\rangle=-c|x^\perp|^2=-\zeta.
\]
Define
\[
        \Phi_{x^\perp}:=\langle x^\perp,\Phi\rangle.
\]
Since $\Phi=A-k^{-1}\vec H\otimes g$ and $\vec H=-cx^\perp$, we have
\begin{equation}\label{eq:Phi-xperp-is-traceless-part}
        \Phi_{x^\perp}=B+\frac{\zeta}{k}g,
        \qquad
        \tr \Phi_{x^\perp}=0.
\end{equation}
Thus $\Phi_{x^\perp}$ is exactly the traceless part of the scalar form $B$.  Consequently,
\begin{equation}\label{eq:Phi-xperp-norm-identity}
        |\Phi_{x^\perp}|^2
        =|B|^2-\frac{\zeta^2}{k}
        =|A_{x^\perp}|^2-\frac{c^2}{k}|x^\perp|^4.
\end{equation}
Moreover, \eqref{eq:gaussian-hess-decomp} becomes
\begin{equation}\label{eq:hess-in-Phi-xperp}
        \Hess_M r
        =\left(1-\frac{\zeta}{k}\right)g+\Phi_{x^\perp}.
\end{equation}
By Lemma~\ref{lem:traceless-est},
\[
        \lambda_{\min}(\Phi_{x^\perp})
        \ge -\sqrt{\frac{k-1}{k}}\,|\Phi_{x^\perp}|.
\]
Since $cR^2\le k$ gives $1-\zeta/k\ge0$, it is enough to require
\begin{equation}\label{eq:suff-Phi-xperp}
        |\Phi_{x^\perp}|^2
        \le
        \frac{k}{k-1}\left(1-\frac{\zeta}{k}\right)^2.
\end{equation}
Using \eqref{eq:Phi-xperp-norm-identity}, condition \eqref{eq:suff-Phi-xperp} is equivalent to
\[
        |A_{x^\perp}|^2
        \le
        \frac{k}{k-1}\left(1-\frac{\zeta}{k}\right)^2+\frac{\zeta^2}{k}.
\]
The right-hand side simplifies exactly as
\[
        \frac{k}{k-1}\left(1-\frac{\zeta}{k}\right)^2+\frac{\zeta^2}{k}
        =1+\frac{(1-\zeta)^2}{k-1}
        =1+\frac{1}{k-1}\left(1-c|x^\perp|^2\right)^2.
\]
This is precisely \eqref{eq:pinch-prop}.  Hence \eqref{eq:suff-Phi-xperp} holds and \eqref{eq:hess-in-Phi-xperp} gives $\Hess_M r\ge0$.
\end{proof}

\begin{remark}
At a fixed point with $\zeta=c|x^\perp|^2$, the condition
\[
        |A_{x^\perp}|^2\le 1+\frac{(1-\zeta)^2}{k-1}
\]
is the sharp Frobenius-norm sufficient condition for the implication
$\lambda_{\min}(A_{x^\perp})\ge -1$, hence for $\Hess_M r=g+A_{x^\perp}\ge0$, when only $\tr A_{x^\perp}=-\zeta$ is prescribed.  Equality occurs for the spectrum
\[
        \operatorname{Spec}(A_{x^\perp})=
        \left\{-1,\frac{1-\zeta}{k-1},\ldots,\frac{1-\zeta}{k-1}\right\}.
\]
Thus the improvement is algebraically lossless at the level of scalar radial projection.
\end{remark}

\section{The minimum set and the radial nullity estimate}\label{sec:nullity}

Let
\begin{equation}\label{eq:C-def}
        C(M)=\{x\in M:r(x)=m_0\},
        \qquad
        m_0=\min_M r.
\end{equation}
Since $r=R^2/2$ on $\partial M$ and
\begin{equation}\label{eq:boundary-derivative-r}
        \partial_\eta r=\langle x,\eta\rangle=R>0
        \qquad\text{on }\partial M,
\end{equation}
a boundary point has strictly negative inward derivative for $r$, so it cannot be a minimum.  Hence
\begin{equation}\label{eq:C-interior}
        C(M)\subset \Int(M).
\end{equation}

The key algebraic ingredient is the radial nullity estimate.

\begin{lemma}[Radial nullity estimate]\label{lem:radial-nullity}
Assume $cR^2\le k$ and the radial pinching condition \eqref{eq:main-pinching}.  For every $p\in C(M)$,
\begin{equation}\label{eq:nullity-leq-one}
        \dim\ker \Hess_M r(p)\le1.
\end{equation}
\end{lemma}

\begin{proof}
At $p\in C(M)$ we have $\nabla^M r=x^T=0$, hence $x=x^\perp$.  Set
\[
        \zeta=c|x^\perp|^2,
        \qquad
        B=A_{x^\perp}.
\]
Since $p\in\Int(M)$, we have $|x|<R$, and from $cR^2\le k$ we get
\begin{equation}\label{eq:zeta-less-k}
        \zeta<k.
\end{equation}
The trace identity is
\begin{equation}\label{eq:B-trace-nullity}
        \tr B=\langle x^\perp,\vec H\rangle=-\zeta.
\end{equation}
Also, at $p$,
\begin{equation}\label{eq:hess-g-B-nullity}
        \Hess_M r=g+B.
\end{equation}

Let
\[
        m=\dim\ker\Hess_M r(p).
\]
We first rule out $m=k$.  If $m=k$, then \eqref{eq:hess-g-B-nullity} gives $B=-g$, hence $\tr B=-k$.  This contradicts \eqref{eq:B-trace-nullity} and \eqref{eq:zeta-less-k}.  Thus $m\le k-1$.

Assume, for contradiction, that $2\le m\le k-1$.  Choose an orthonormal basis $e_1,\ldots,e_k$ diagonalizing $\Hess_M r(p)$, with $e_1,\ldots,e_m$ spanning its kernel.  Since $B=\Hess_M r-g$, the same basis diagonalizes $B$.  Write
\[
        \alpha_j=B(e_j,e_j)=\langle x^\perp,A(e_j,e_j)\rangle.
\]
For $1\le i\le m$, \eqref{eq:hess-g-B-nullity} gives
\[
        \alpha_i=-1.
\]
The trace identity \eqref{eq:B-trace-nullity} gives
\[
        \sum_{j=m+1}^k\alpha_j=m-\zeta.
\]
Because the basis diagonalizes $B$, there is no off-diagonal loss:
\begin{equation}\label{eq:Axperp-exact-diagonal}
        |A_{x^\perp}|^2=|B|^2
        =\sum_{i,j=1}^k B(e_i,e_j)^2
        =\sum_{j=1}^k\alpha_j^2.
\end{equation}
Therefore Cauchy--Schwarz gives the sharp lower bound
\begin{equation}\label{eq:nullity-lower}
        |A_{x^\perp}|^2
        =\sum_{j=1}^k\alpha_j^2
        \ge m+\frac{(m-\zeta)^2}{k-m}.
\end{equation}
The radial pinching condition gives
\begin{equation}\label{eq:nullity-upper}
        |A_{x^\perp}|^2
        \le 1+\frac{(1-\zeta)^2}{k-1}.
\end{equation}
Subtracting \eqref{eq:nullity-upper} from \eqref{eq:nullity-lower}, we obtain the exact factorization
\begin{align*}
\Delta
&=\left(m+\frac{(m-\zeta)^2}{k-m}\right)
  -\left(1+\frac{(1-\zeta)^2}{k-1}\right)  \\
&=\frac{(m-1)(k-\zeta)^2}{(k-1)(k-m)}.
\end{align*}
Since $m\ge2$ and $\zeta<k$, this gives $\Delta>0$, contradicting \eqref{eq:nullity-lower} and \eqref{eq:nullity-upper}.  Therefore $m\le1$.
\end{proof}

\begin{corollary}[Pointwise equality and transverse nondegeneracy]\label{cor:equality-branch}
Under the hypotheses of Theorem~\ref{thm:mainA}, let $p\in C(M)$, set
\[
        \zeta=c|p|^2,
\]
and suppose that
\[
        \dim\ker\Hess_M r(p)=1.
\]
Then the radial pinching is saturated at $p$:
\[
        |A_{x^\perp}|^2(p)=1+\frac{(1-\zeta)^2}{k-1}.
\]
Moreover,
\[
        \operatorname{Spec}\bigl(A_{x^\perp}|_p\bigr)
        =\left\{-1,\frac{1-\zeta}{k-1},\ldots,
        \frac{1-\zeta}{k-1}\right\},
\]
and hence
\[
        \operatorname{Spec}\bigl(\Hess_M r|_p\bigr)
        =\left\{0,\frac{k-\zeta}{k-1},\ldots,
        \frac{k-\zeta}{k-1}\right\}.
\]
In particular, $\Hess_M r(p)$ is positive definite on
$\bigl(\ker\Hess_M r(p)\bigr)^\perp$.
\end{corollary}

\begin{proof}
Set $B=A_{x^\perp}|_p$.  Since $p\in C(M)$, we have $x=x^\perp$ at $p$, and
\[
        \tr B=\langle x^\perp,\vec H\rangle=-\zeta,
        \qquad
        \Hess_M r(p)=g+B.
\]
Choose an orthonormal basis $e_1,\ldots,e_k$ diagonalizing $B$, with
$e_1$ spanning $\ker\Hess_M r(p)$, and write
$\alpha_j=B(e_j,e_j)$.  Then $\alpha_1=-1$ and
\[
        \sum_{j=2}^k\alpha_j=1-\zeta.
\]
Therefore Cauchy--Schwarz gives
\[
        |A_{x^\perp}|^2(p)
        =1+\sum_{j=2}^k\alpha_j^2
        \ge 1+\frac{(1-\zeta)^2}{k-1}.
\]
The radial pinching gives the reverse inequality, so equality holds throughout.
Equality in Cauchy--Schwarz forces
$\alpha_2=\cdots=\alpha_k=(1-\zeta)/(k-1)$, proving the asserted spectrum of
$A_{x^\perp}|_p$.  Adding the identity gives the spectrum of $\Hess_M r|_p$.
Finally, $p\in\Int(M)$ and $cR^2\le k$, so $\zeta<k$; hence the transverse
eigenvalue $(k-\zeta)/(k-1)$ is strictly positive.
\end{proof}

\section{Analytic structure and sublevel topology}\label{sec:topology}

This section proves the topological part of Theorem~\ref{thm:mainA}.  The argument has four steps: a boundary barrier, the equality $\Crit(r)=C(M)$, analytic structure of the minimum set, and sublevel isotopy.

\begin{lemma}[Boundary barrier and interior geodesics]\label{lem:interior-geodesic}
Let $M^k\subset \Bcl_R^N$ be a compact free boundary submanifold.  If
$p,q\in\Int(M)$ lie in the same connected component of $M$, then a
length-minimizing curve joining $p$ to $q$ may be chosen to be a smooth
geodesic contained in $\Int(M)$.
\end{lemma}

\begin{proof}
If $p=q$, the constant curve gives the conclusion, so assume $p\ne q$.
The connected component of $M$ containing $p$ and $q$ is compact as a
length space, so the intrinsic distance between them is realized by a curve
$\gamma:[0,L]\to M$, which we parametrize by arclength.  The regularity needed
here is precisely the conclusion of Alexander--Alexander
\cite[Theorem~1(A), p.~481]{AlexanderAlexander1981}: every shortest path in a
Riemannian $C^3$ manifold with $C^1$ boundary is of class $C^1$.  Since both
the induced metric on $M$ and $\partial M$ are smooth, that theorem applies to
$\gamma$.  On every open parameter interval mapped into $\Int(M)$, the usual
first-variation argument shows that $\gamma$ is a smooth Riemannian geodesic.
For the finer obstacle-contact structure, including the possible failure of
$C^2$ regularity, see
\cite[Section~2, pp.~169--171]{AlexanderBergBishop1987}.  The argument below
uses only the $C^1$ regularity across a contact time and one-sided
second-derivative information on its interior side.

Consider the ambient squared-distance function
\[
        r(x)=\frac12|x|^2
\]
restricted to $M$.  Since $M\subset\Bcl_R^N$,
\[
        r\le\frac{R^2}{2}\quad\text{on }M,
        \qquad
        r=\frac{R^2}{2}\quad\text{on }\partial M.
\]

Suppose, for contradiction, that the contact set
\[
        I=\{t\in[0,L]:\gamma(t)\in\partial M\}
\]
is nonempty.  Since $p,q\in\Int(M)$, the set $I$ is a compact subset of
$(0,L)$.  Let $t_0=\min I$.  Then
\[
        \gamma(t_0)\in\partial M,
        \qquad
        \gamma\bigl([0,t_0)\bigr)\subset\Int(M),
\]
so $\gamma$ is a smooth interior geodesic on $(0,t_0)$.

Set
\[
        \varphi=r\circ\gamma.
\]
Because $\gamma$ is $C^1$, the function $\varphi$ is $C^1$.  Moreover,
\[
        \varphi\le\frac{R^2}{2},
        \qquad
        \varphi(t_0)=\frac{R^2}{2}.
\]
Thus $t_0$ is an interior maximum point of $\varphi$, and hence
\[
        \varphi'(t_0)=0.
\]

For $t\in(0,t_0)$, Lemma~\ref{lem:hess-identity} and
$|\dot\gamma(t)|=1$ give
\[
\begin{aligned}
        \varphi''(t)
        &=\Hess_M r\bigl(\dot\gamma(t),\dot\gamma(t)\bigr)\\
        &=1+
        \left\langle
        x^\perp(\gamma(t)),
        A_{\gamma(t)}\bigl(\dot\gamma(t),\dot\gamma(t)\bigr)
        \right\rangle .
\end{aligned}
\]
The normal projection $x^\perp$ is smooth on $M$ and, by
\eqref{eq:xperp-boundary}, vanishes on $\partial M$.  Since $M$ is compact
and $\gamma$ has unit speed,
\[
\left|
\left\langle
x^\perp(\gamma(t)),
A_{\gamma(t)}\bigl(\dot\gamma(t),\dot\gamma(t)\bigr)
\right\rangle
\right|
\le
|x^\perp(\gamma(t))|\,\|A\|_{C^0(M)}
\longrightarrow0
\]
as $t\uparrow t_0$.  Therefore
\[
        \varphi''(t)\longrightarrow1
        \qquad\text{as }t\uparrow t_0.
\]
Hence there exists $\delta>0$ such that
\[
        \varphi''(t)\ge\frac12
        \qquad
        \text{for all }t\in(t_0-\delta,t_0).
\]

For every $t\in(t_0-\delta,t_0)$, the mean value theorem applied to
$\varphi'$ gives some $\xi\in(t,t_0)$ such that
\[
        \varphi'(t)
        =\varphi'(t_0)-\varphi''(\xi)(t_0-t)
        \le-\frac{t_0-t}{2}.
\]
Consequently,
\[
\begin{aligned}
        \varphi(t)-\varphi(t_0)
        &=-\int_t^{t_0}\varphi'(s)\,ds\\
        &\ge
        \frac12\int_t^{t_0}(t_0-s)\,ds
        =\frac14(t_0-t)^2>0.
\end{aligned}
\]
Thus
\[
        \varphi(t)>\varphi(t_0)=\frac{R^2}{2},
\]
contradicting $\varphi\le R^2/2$ on $M$.

Therefore $I=\varnothing$.  The minimizing curve is contained in
$\Int(M)$, where it is a smooth geodesic.
\end{proof}

\begin{lemma}[Critical set equals minimum set]\label{lem:crit-equals-C}
Under the hypotheses of Theorem~\ref{thm:mainA},
\begin{equation}\label{eq:crit-equals-C}
        \Crit(r)=C(M).
\end{equation}
Moreover, $C(M)$ is compact, connected and totally convex.
\end{lemma}

\begin{proof}
By Proposition~\ref{prop:pinch-hess}, $\Hess_M r\ge0$.  Let $q\in\Crit(r)$ and let $p\in C(M)$.  Both points lie in the interior: $p$ by \eqref{eq:C-interior}, and $q$ because $\nabla^Mr=R\eta\ne0$ along $\partial M$.  By Lemma~\ref{lem:interior-geodesic}, choose an interior minimizing geodesic $\gamma:[0,1]\to M$ joining $p$ to $q$.  Then
\[
        \varphi(t)=r(\gamma(t))
\]
is convex, because
\[
        \varphi''(t)=\Hess_M r(\dot\gamma,\dot\gamma)\ge0.
\]
Since $p,q$ are critical points of $r$, we have $\varphi'(0)=\varphi'(1)=0$.  A convex function has nondecreasing derivative, hence $\varphi'\equiv0$ and $\varphi$ is constant.  Thus $q\in C(M)$.  This proves $\Crit(r)\subset C(M)$, and the reverse inclusion is automatic.

The same argument applied to two points of $C(M)$ shows that any minimizing geodesic joining them lies in $C(M)$.  Hence $C(M)$ is totally convex.  This also gives connectedness.  Compactness follows because $M$ is compact and $C(M)$ is closed.
\end{proof}

\begin{lemma}[Analytic structure of the minimum set]\label{lem:analytic-C}
Under the hypotheses of Theorem~\ref{thm:mainA}, the set $C(M)$ is either a single point or a smooth closed curve diffeomorphic to $\Sph^1$.
\end{lemma}

\begin{proof}
We first record the analytic regularity input in a fixed elliptic gauge.  At an interior point, after a rigid motion, write $M$ locally as a Euclidean graph
\[
        F(y)=(y,u(y)),\qquad y\in U\subset \R^k,
\]
over its tangent plane.  In this graph, equivalently normal-graph, gauge the equation $\vec H=-cF^\perp$ becomes a strongly elliptic quasilinear system
\[
        g^{ij}(Du)\,\partial_{ij}u^\alpha
        =\mathcal F^\alpha(y,u,Du),
        \qquad \alpha=1,\ldots,N-k,
\]
whose coefficients are real analytic in $(y,u,Du)$ and whose principal matrix $(g^{ij})$ is positive definite.  Thus the usual degeneracy coming from reparametrization invariance has been removed.  Morrey's analytic regularity theorem for analytic elliptic systems \cite[Chapter~6]{Morrey1966} implies that the immersion is real analytic in $\Int(M)$.  Therefore $r=|x|^2/2$ restricted to $M$ is real analytic in a neighborhood of $C(M)$.

Fix $p\in C(M)$.  Lemma~\ref{lem:radial-nullity} gives
\[
        \dim\ker \Hess_M r(p)\le1.
\]
If the nullity is zero, then $p$ is an isolated local minimum.  Suppose now that the nullity is one.  Split the tangent space at $p$ as
\[
        T_pM=K\oplus E,
        \qquad K=\ker\Hess_M r(p),\quad \dim K=1,
\]
where the Hessian is positive definite on $E$.  Choose analytic coordinates
$(u,y)\in K\oplus E$ centered at $p$.  Apply the analytic Morse lemma with
parameters in the precise form of \cite[Theorem~4]{Feehan2020}, taking $u$
as the parameter and $y$ as the nondegenerate variable.  Its hypotheses hold
because $dr(p)=0$ and the $E\times E$ block of $\Hess_Mr(p)$ is positive
definite.  The theorem gives a parameter-preserving real-analytic change of
coordinates for which the $y$-dependence is its fixed nondegenerate quadratic
form; after an analytic linear normalization of that positive form, one obtains
real-analytic coordinates
\[
        (u,y_1,\ldots,y_{k-1})
\]
centered at $p$ such that
\begin{equation}\label{eq:analytic-splitting}
        r(u,y)=r(p)+|y|^2+h(u),
\end{equation}
where $h$ is a real-analytic function of one variable.  Since $p\in C(M)$, $r-r(p)\ge0$ locally; hence $h(u)\ge0$ after restricting the coordinate neighborhood.  Also $h(0)=h'(0)=h''(0)=0$, because the $u$-direction is the null direction at $p$.  In these coordinates,
\[
        C(M)\cap U=\{y=0,\ h(u)=0\}.
\]
The zero set of a real-analytic function of one variable is either discrete or locally the whole interval.  Hence $C(M)$ is locally either an isolated point or a smooth analytic curve.  In particular, it cannot have endpoints of the form $[0,\varepsilon)$, and it cannot have branches.

Since $C(M)$ is connected, if it contains more than one point it cannot contain an isolated point.  Thus it is a compact connected one-dimensional smooth manifold without boundary.  Therefore $C(M)\cong\Sph^1$.
\end{proof}

\begin{proposition}[Sublevel isotopy and topology]\label{prop:sublevel-topology}
Under the hypotheses of Theorem~\ref{thm:mainA}, the topology is determined by
the minimum set:
\[
        C(M)=\{p_0\}
        \quad\Longrightarrow\quad
        M\cong\D^k,
\]
and
\[
        C(M)\cong\Sph^1
        \quad\Longrightarrow\quad
        M\cong\Sph^1\times\D^{k-1}.
\]
Consequently, $M\cong\D^k$ or
$M\cong\Sph^1\times\D^{k-1}$.
\end{proposition}

\begin{proof}
Let $m_0=\min_M r$.  By Lemma~\ref{lem:crit-equals-C}, $\Crit(r)=C(M)$, so $\nabla r\ne0$ on $M\setminus C(M)$.  Choose $\varepsilon>0$ sufficiently small so that $m_0+\varepsilon$ is a regular value of $r$, and put
\[
        M_\varepsilon=\{x\in M:r(x)\le m_0+\varepsilon\}.
\]
On the region $\{m_0+\varepsilon\le r\le R^2/2\}$, the vector field
\[
        X=-\frac{\nabla r}{|\nabla r|^2}
\]
is smooth.  Along $\partial M$, using \eqref{eq:boundary-derivative-r},
\[
        \langle X,\eta\rangle
        =-\frac{\partial_\eta r}{|\nabla r|^2}<0.
\]
Thus the flow of $X$ points into $M$ along the boundary and cannot escape.  Moreover,
\[
        \frac{d}{dt}r(\Phi_t(x))=-1.
\]
It follows that the outer region is a product, so $M$ is diffeomorphic to $M_\varepsilon$.

We now identify $M_\varepsilon$.  Since $\bigcap_{\varepsilon>0}M_\varepsilon=C(M)$ and the sets $M_\varepsilon$ are compact and nested, for every neighborhood $W$ of $C(M)$ there is $\varepsilon_0>0$ with $M_\varepsilon\subset W$ for all $\varepsilon\le\varepsilon_0$.

\smallskip
\noindent\emph{Case 1: $C(M)=\{p\}$.}  Shrinking $\varepsilon$, we may assume that $M_\varepsilon$ is contained in the analytic coordinate neighborhood of Lemma~\ref{lem:analytic-C} centered at $p$.  If $\Hess_Mr(p)$ is nondegenerate, the classical Morse lemma \cite{Milnor1963} gives coordinates $w$ in which $r=m_0+|w|^2$, so $M_\varepsilon=\{|w|^2\le\varepsilon\}$ is a smooth closed $k$-ball.  If the nullity is one, then in the normal form \eqref{eq:analytic-splitting} the real-analytic function $h\ge0$ has an isolated zero at $0$: otherwise $h\equiv0$ near $0$ and $C(M)$ would contain a curve through $p$.  Hence $h$ vanishes at $0$ to finite, necessarily even, order, and since $h''(0)=0$,
\[
        h(u)=u^{2\ell}a(u),
        \qquad \ell\ge2,
        \qquad a\ \text{real analytic},\quad a(0)>0 .
\]
The substitution
\[
        v=u\,a(u)^{1/(2\ell)},
        \qquad
        \frac{dv}{du}(0)=a(0)^{1/(2\ell)}>0,
\]
is a real-analytic diffeomorphism near $0$, because a positive real-analytic function raised to a fixed real power is real analytic; in the coordinates $(v,y)$ one has $h=v^{2\ell}$ and hence
\[
        M_\varepsilon=\{G\le\varepsilon\},
        \qquad
        G(v,y)=v^{2\ell}+|y|^2 .
\]
The function $G$ is smooth and strictly convex on $\R^k$ (although its
Hessian degenerates in the $v$-direction at $v=0$), $G(0)=0$, and
$\nabla G$ vanishes only at the origin.  Hence $\{G\le\varepsilon\}$ is a
compact convex body with nonempty interior and smooth boundary
$\{G=\varepsilon\}$.  We now use the standard smooth-convex-body theorem:
every compact convex body with nonempty interior and smooth boundary is
diffeomorphic, as a smooth manifold with boundary, to the standard closed
ball.  Therefore $M_\varepsilon$, and hence $M$, is diffeomorphic to $\D^k$.

\smallskip
\noindent\emph{Case 2: $C(M)\cong\Sph^1$.}  By Lemma~\ref{lem:radial-nullity}, along $C(M)$ the kernel of $\Hess_Mr$ is exactly the tangent line of $C(M)$.  The normal nondegeneracy is explicit from Corollary~\ref{cor:equality-branch}: if $p\in C(M)$ and $\zeta=c|p|^2$, then the nonzero eigenvalue of $\Hess_Mr=g+A_{x^\perp}$ is
\[
        1+\frac{1-\zeta}{k-1}=\frac{k-\zeta}{k-1}.
\]
Since $p\in\Int(M)$, $|p|<R$; thus $\zeta<k$ (for $c>0$, $\zeta<cR^2\le k$, while for $c=0$, $\zeta=0<k$), and the displayed eigenvalue is strictly positive.  Moreover, in the local model \eqref{eq:analytic-splitting} centered at any point of $C(M)$, the circle case forces $h\equiv0$, so $r=m_0+|y|^2$ there.  Thus $r$ is a Morse--Bott function whose critical submanifold $C(M)$ is compact of index zero.  The precise local normal form is the Morse--Bott lemma of \cite[Theorem~2.10]{Feehan2020}; its standard global tubular-neighborhood consequence for a compact critical submanifold gives a neighborhood $\nu_\delta$ of the zero section of the normal bundle $\nu=\nu\bigl(C(M)\subset M\bigr)$, a fiber metric $|\cdot|$ on $\nu$, and a diffeomorphism $\Phi:\nu_\delta\to\Phi(\nu_\delta)\subset M$ onto a tubular neighborhood of $C(M)$ with
\[
        r\circ\Phi=m_0+|\xi|^2 .
\]
Shrinking $\varepsilon$ so that $M_\varepsilon\subset\Phi(\nu_\delta)$, we conclude that
\[
        M_\varepsilon=\Phi\bigl(\{\xi\in\nu:|\xi|^2\le\varepsilon\}\bigr)
\]
is the closed disk bundle of $\nu$.  Since $M$ is orientable and $C(M)\cong\Sph^1$ is orientable, the normal bundle is orientable; equivalently,
\[
        w_1(\nu)=w_1(TM|_{C(M)})-w_1(TC(M))=0.
\]
Every orientable real vector bundle over $\Sph^1$ is trivial.  Hence
\[
        M_\varepsilon\cong \Sph^1\times\D^{k-1},
\]
and the result follows.
\end{proof}

\begin{remark}[Role of orientability]\label{rem:orientability}
The orientability assumption in Theorem~\ref{thm:mainA} is used only in the last step of Proposition~\ref{prop:sublevel-topology}, where it trivializes the normal disk bundle over the minimum circle.  The Hessian convexity, the radial nullity estimate, the analytic structure of $C(M)$, and the sublevel product structure do not require orientability.  Without orientability, the circle case gives a closed disk bundle over $\Sph^1$; orientability rules out the twisted bundle and yields $\Sph^1\times \D^{k-1}$.
\end{remark}

\begin{proof}[Proof of Theorem~\ref{thm:mainA}]
Proposition~\ref{prop:pinch-hess} gives $\Hess_M r\ge0$.
Lemmas~\ref{lem:radial-nullity}, \ref{lem:crit-equals-C}, and
\ref{lem:analytic-C}, together with
Proposition~\ref{prop:sublevel-topology}, give the two topological alternatives.
If the pinching is strict on $M$, the one-dimensional minimum-set alternative is
impossible.  At each $p\in C(M)$ its tangent line lies in
$\ker\Hess_Mr(p)$; Lemma~\ref{lem:radial-nullity} shows that this kernel is
one-dimensional, and Corollary~\ref{cor:equality-branch} would force equality
in \eqref{eq:main-pinching} at $p$.  Hence $C(M)$ is a point, and
Proposition~\ref{prop:sublevel-topology} gives $M\cong\D^k$.
\end{proof}

\section{Rotational rigidity in dimension two}\label{sec:surface-rigidity}

Let $\Sigma^2\subset\Bcl_R^3$ satisfy the hypotheses of
Theorem~\ref{thm:mainA}, and assume that $\Sigma$ is not diffeomorphic to
$\D^2$.  By Lemma~\ref{lem:analytic-C} and
Proposition~\ref{prop:sublevel-topology}, its minimum set satisfies
$C(\Sigma)\cong\Sph^1$.  We fix a unit normal $\nu$ along $\Sigma$.

\begin{proposition}[Roundness of the minimum circle]\label{prop:roundness}
The curve $C(\Sigma)$ is a great circle of a sphere centered at the origin.
\end{proposition}

\begin{proof}
Along $C(\Sigma)$,
\[
        \nabla^\Sigma r=x^T=0,
\]
so $x$ is everywhere normal to $\Sigma$.  Let
\[
        \rho=|x|=\sqrt{2m_0}
\]
on $C(\Sigma)$, and let $T$ be the unit tangent to $C(\Sigma)$.  Since $C(\Sigma)$ is a smooth one-dimensional totally convex subset, it is a geodesic in $\Sigma$, and $T\in\ker\Hess_\Sigma r$.  Therefore
\[
        \nabla_T^\Sigma T=0,
        \qquad
        \bar\nabla_TT=A(T,T).
\]
Moreover,
\[
        0=\Hess_\Sigma r(T,T)=1+\langle x,A(T,T)\rangle.
\]
Since $x$ spans the normal line of $\Sigma$ along $C(\Sigma)$, $A(T,T)$ is parallel to $x$, and hence
\[
        A(T,T)=-\frac{x}{\rho^2}.
\]
Thus
\[
        \bar\nabla_TT=-\frac{x}{\rho^2},
\]
which is precisely the equation of a great circle on the sphere $\Sph^2_\rho$ centered at the origin.
\end{proof}

\begin{lemma}[Principal directions along the minimum circle]\label{lem:principal-directions}
Let $V$ be a unit vector tangent to $\Sigma$ and orthogonal to $T$ along $C(\Sigma)$.  Then
\[
        A(T,V)=0
\]
along $C(\Sigma)$.  Hence $T$ and $V$ are principal directions.
\end{lemma}

\begin{proof}
Since $\Hess_\Sigma r\ge0$ and $T\in\ker\Hess_\Sigma r$, we have
\[
        \Hess_\Sigma r(T,W)=0
\]
for every tangent vector $W$.  Taking $W=V$ gives
\[
        0=\Hess_\Sigma r(T,V)
        =\langle T,V\rangle+\langle x,A(T,V)\rangle
        =\langle x,A(T,V)\rangle.
\]
Because $x$ is normal to $\Sigma$ and nonzero along $C(\Sigma)$, the last equality implies $A(T,V)=0$.
\end{proof}

\begin{proposition}[First-jet vanishing of the rotational Jacobi field]\label{prop:firstjet}
Let $K$ be the rotational Killing field around the axis orthogonal to the plane of $C(\Sigma)$, and set
\[
        u=\langle K,\nu\rangle.
\]
Then
\[
        u=0,
        \qquad
        \nabla u=0
        \quad\text{along }C(\Sigma).
\]
Moreover,
\begin{equation}\label{eq:weighted-jacobi}
        \Delta_f u+(|A|^2+c)u=0.
\end{equation}
\end{proposition}

\begin{proof}
By Proposition~\ref{prop:roundness}, $C(\Sigma)$ is a great circle of the sphere $\Sph^2_\rho$ centered at the origin.  Therefore the axis orthogonal to the plane of $C(\Sigma)$ passes through the origin.  The associated rotational Killing field is generated by a constant skew-symmetric matrix $\Omega\in\mathfrak{so}(3)$,
\[
        K(y)=\Omega y.
\]
In particular, since $f(y)=c|y|^2/2$ is radial, $K(f)=0$.

The one-parameter group generated by $K$ preserves the Euclidean metric, the closed ball $\Bcl_R^3$ together with its boundary sphere, and the Gaussian weight.  Hence it maps $f$-minimal surfaces to $f$-minimal surfaces.  Differentiating the $f$-minimal equation along this variation gives the weighted Jacobi equation
\[
        \Delta_f u+(|A|^2+\Ric_f(\nu,\nu))u=0,
\]
where $\Ric_f:=\Ric+\Hess f$.  In the Euclidean Gaussian space,
$\Ric_f(\nu,\nu)=\Hess f(\nu,\nu)=c$, giving \eqref{eq:weighted-jacobi}.

Along $C(\Sigma)$, the vector $K$ is tangent to the great circle, hence parallel to $T$, while $\nu$ is parallel to $x$.  Thus $u=\langle K,\nu\rangle=0$ on $C(\Sigma)$.  Therefore $du(T)=0$.

It remains to compute $du(V)$.  Since $K(y)=\Omega y$,
\begin{align*}
        du(V)
        &=V\langle K,\nu\rangle  \\
        &=\langle \bar\nabla_VK,\nu\rangle+\langle K,\bar\nabla_V\nu\rangle \\
        &=\langle \Omega V,\nu\rangle+\langle K,-S(V)\rangle.
\end{align*}
By Lemma~\ref{lem:principal-directions}, $V$ is a principal direction, so $S(V)=\kappa_V V$.  Since $K\parallel T$, we have $\langle K,S(V)\rangle=0$.  Also $\nu=\pm x/\rho$, hence
\[
        \Omega \nu=\pm \frac1\rho \Omega x=\pm \frac1\rho K\parallel T.
\]
Using skew-symmetry of $\Omega$,
\[
        \langle \Omega V,\nu\rangle=-\langle V,\Omega \nu\rangle=0.
\]
Thus $du(V)=0$, and $\nabla u=0$ on $C(\Sigma)$.
\end{proof}

\begin{proposition}[Nodal-set conclusion]\label{prop:nodal}
The function $u=\langle K,\nu\rangle$ vanishes identically on $\Sigma$.
\end{proposition}

\begin{proof}
Set $h:=f|_{\Sigma}$ and write
\[
        u=e^{h/2}v
        \qquad\text{(equivalently, }v=e^{-h/2}u\text{)}.
\]
A direct calculation, with all differential operators taken on $\Sigma$, gives
\begin{align*}
        \Delta_f u
        &=\Delta u-\langle\nabla h,\nabla u\rangle \\
        &=e^{h/2}\left(
          \Delta v+
          \left(\frac12\Delta h-\frac14|\nabla h|^2\right)v
        \right).
\end{align*}
Consequently, \eqref{eq:weighted-jacobi} is equivalent to the
Schr\"odinger equation
\begin{equation}\label{eq:conjugated-jacobi}
        \Delta v+Vv=0,
        \qquad
        V:=|A|^2+c+\frac12\Delta h-\frac14|\nabla h|^2,
\end{equation}
whose potential $V$ is smooth.  Since $e^{h/2}>0$, the two functions have the
same nodal set.  Moreover,
\[
        \nabla u=e^{h/2}\left(\nabla v+\frac12v\nabla h\right),
\]
so at a nodal point one has $\nabla u=e^{h/2}\nabla v$.  Hence
\[
        \{u=0,\ \nabla u=0\}
        =\{v=0,\ \nabla v=0\}.
\]

Assume now that $u\not\equiv0$; then $v\not\equiv0$.  By
\eqref{eq:C-interior}, the minimum circle $C(\Sigma)$ lies in
$\Int(\Sigma)$.  Cheng's local nodal-set theorem for a nontrivial solution of
a Schr\"odinger equation on a Riemannian surface
\cite[Theorem~2.5]{Cheng1976} implies that its critical nodal points are
isolated.  On the other hand, Proposition~\ref{prop:firstjet} and the equality
of the critical nodal sets above give
\[
        C(\Sigma)\subset\{v=0,\ \nabla v=0\}.
\]
This is impossible because $C(\Sigma)\cong\Sph^1$ has no isolated points.
Therefore $u\equiv0$ on $\Sigma$.
\end{proof}

\begin{proof}[Proof of Theorem~\ref{thm:mainB}]
By Proposition~\ref{prop:nodal}, $\langle K,\nu\rangle\equiv0$.  Thus the
rotational Killing field $K$ is everywhere tangent to the immersion, and its
immersed image is invariant under the one-parameter rotation group generated by
$K$.
\end{proof}

\begin{remark}[Equality along the minimum circle]
By Corollary~\ref{cor:equality-branch}, the annular case is an equality
branch of the radial pinching.  Along $C(\Sigma)$ the scalar form $A_{x^\perp}$ has spectrum
\[
        \left\{-1,\,1-c|x|^2\right\}
\]
when $k=2$.  Thus strict radial pinching rules out the rotational annulus before the nodal-set argument is used.
Theorem~\ref{thm:local-gaussian-family} shows that this equality alternative
is realized by genuinely Gaussian examples for all sufficiently small
positive values of $c$ (equivalently, after scaling, of $cR^2$).
\end{remark}

\section{Rotational ODE, a perturbative Gaussian family, and the global shooting problem}\label{sec:ode}

We now record the ODE for rotational Gaussian free boundary $f$-minimal
annuli.  Rotational and low-cohomogeneity free boundary minimal
hypersurfaces, together with related rotational-domain uniqueness results,
provide the unweighted model background
\cite{FreidinGulianMcGrath2017,BarbosaViana2022,BarbosaFreitasMeloVitorio2023};
the system below is the corresponding radial Gaussian $f$-minimal profile
equation.  In contrast with a purely conditional shooting formulation, we
prove here that the nondegenerate unweighted shooting root generates a
locally unique Gaussian family, including the unweighted endpoint $c=0$,
for all sufficiently small $c\ge0$, and that every member satisfies the
radial pinching.  The remaining issue isolated
below is global uniqueness of the free boundary shooting root in the full
shooting domain: the implicit-function argument gives uniqueness only near
the continued critical-catenoid root.

Let the surface be obtained by rotating an arclength-parametrized profile curve
\[
        s\mapsto (\rho(s),0,z(s)),
        \qquad \rho(s)>0,
\]
about the $z$-axis.  Here $\rho$ is the scalar radial coordinate in the meridian half-plane; the ambient position vector is denoted by $X$.  Write
\[
        \rho'(s)=\cos\theta(s),
        \qquad
        z'(s)=\sin\theta(s).
\]
Choose the unit normal
\[
        \bn=(-\sin\theta\cos\varphi,
              -\sin\theta\sin\varphi,
               \cos\theta).
\]
Then
\[
        X\cdot \bn=-\rho\sin\theta+z\cos\theta,
\]
and the two principal curvatures, with respect to this orientation and our trace convention, are
\[
        \kappa_1=\theta',
        \qquad
        \kappa_2=\frac{\sin\theta}{\rho}.
\]
Thus the scalar form of $\vec H=-cX^\perp$ is
\[
        \theta'+\frac{\sin\theta}{\rho}
        =c(\rho\sin\theta-z\cos\theta).
\]
Equivalently,
\begin{equation}\label{eq:rot-ode}
\begin{cases}
        \rho'=\cos\theta,\\[2mm]
        z'=\sin\theta,\\[2mm]
        \displaystyle \theta'=c(\rho\sin\theta-z\cos\theta)-\frac{\sin\theta}{\rho}.
\end{cases}
\end{equation}
Changing the orientation of the profile normal changes both sides of the scalar mean-curvature equation consistently; the system above fixes one convention.

\begin{lemma}[Minimum-circle neck normalization]\label{lem:minimum-neck-normalization}
Let $\Sigma$ be a rotational annulus arising from Theorem~\ref{thm:mainB}.  After an ambient rotation and a choice of profile orientation, its meridian curve can be parametrized so that it satisfies the neck initial data
\begin{equation}\label{eq:neck-initial}
        \rho(0)=a\in(0,R),
        \qquad
        z(0)=0,
        \qquad
        \theta(0)=\frac\pi2.
\end{equation}
\end{lemma}

\begin{proof}
By Proposition~\ref{prop:roundness}, the minimum circle $C(\Sigma)$ is a great circle of a sphere centered at the origin.  Rotating the ambient coordinates, we may assume that this circle lies in the plane $z=0$ and that the rotational axis is the $z$-axis.  In the meridian half-plane the minimum circle corresponds to a single point $(\rho,z)=(a,0)$ with $a>0$.

The set $C(\Sigma)$ is the minimum set of $r=|X|^2/2$.  Along the profile curve,
\[
        r(s)=\frac12\bigl(\rho(s)^2+z(s)^2\bigr).
\]
At the minimum-circle point $s=0$ this function has a minimum.  Hence
\[
        0=r'(0)=\rho(0)\rho'(0)+z(0)z'(0)=a\rho'(0).
\]
Since $a>0$, we have $\rho'(0)=0$.  As the profile is parametrized by arclength, $\rho'=\cos\theta$, so $\cos\theta(0)=0$.  Reversing the profile orientation if necessary gives $\theta(0)=\pi/2$.
\end{proof}

When $c$ is fixed, we suppress it from the shooting notation.  When the
Gaussian parameter is allowed to vary, we write
$(\rho_{c,a},z_{c,a},\theta_{c,a})$, $S(c,a)$, and
$\widehat{\cQ}(c,a)$ for the corresponding solution, first hitting time,
and boundary residual.

For $a\in(0,R)$, let $(\rho_a,z_a,\theta_a)$ denote the maximal solution of \eqref{eq:rot-ode} with the neck initial data \eqref{eq:neck-initial}, defined on a maximal interval $[0,s_a^+)$ on which $\rho_a>0$, and let
\[
        S_a=\inf\bigl\{s\in[0,s_a^+):\rho_a(s)^2+z_a(s)^2=R^2\bigr\}
        \ \in\ (0,+\infty]
\]
be the first time at which the orbit reaches the sphere of radius $R$, with the convention $\inf\varnothing=+\infty$.  We do not claim that $S_a$ is finite for every $a$, nor that the sphere is reached transversally when it is reached; both questions are part of the shooting problem.  Define the \emph{shooting domain}
\begin{equation}\label{eq:shooting-domain}
        \cA
        =\Bigl\{a\in(0,R):\ S_a<+\infty
        \ \text{ and }\
        \tfrac{d}{ds}\bigl(\rho_a^2+z_a^2\bigr)(S_a)>0\Bigr\}.
\end{equation}
By smooth dependence on initial conditions and the implicit function theorem applied to $(s,a)\mapsto\rho_a(s)^2+z_a(s)^2-R^2$, the set $\cA$ is open in $(0,R)$ and $a\mapsto S_a$ is smooth on $\cA$.  For arbitrary values of $(c,R)$, the set $\cA$ need a priori be neither an interval nor nonempty; Theorem~\ref{thm:local-gaussian-family} shows, however, that it contains a free boundary root when $c>0$ is sufficiently small for fixed $R$.  The free boundary condition is that the profile curve meets the circle of radius $R$ orthogonally, equivalently that the position vector is parallel to the unit tangent at the endpoint; in the meridian plane this says that the position vector has vanishing normal component.  Accordingly, define the shooting map
\begin{equation}\label{eq:shooting-target}
        \cQ:\cA\longrightarrow\R,
        \qquad
        \cQ(a)=\rho_a(S_a)\sin\theta_a(S_a)-z_a(S_a)\cos\theta_a(S_a),
\end{equation}
which is smooth on $\cA$.  A rotational free boundary Gaussian $f$-minimal annulus corresponds exactly to a root
\[
        a\in\cA,
        \qquad
        \cQ(a)=0,
\]
its backward profile branch $s<0$ being the mirror image of the forward branch across $\{z=0\}$, by the reflection symmetry of \eqref{eq:rot-ode} with the data \eqref{eq:neck-initial}.  No rotational solution is lost by restricting to $\cA$: if a free boundary annulus is normalized as in Lemma~\ref{lem:minimum-neck-normalization}, its interior lies in the open ball $\B_R^3$ because $\Sigma\cap\partial\B_R^3=\partial\Sigma$ by the defining condition \eqref{eq:freebdry-contact}, so the boundary parameter of the forward profile arc is the first hitting time $S_a<+\infty$, and there the free boundary condition makes the position vector parallel to the unit tangent, whence
\[
        \frac{d}{ds}\bigl(\rho_a^2+z_a^2\bigr)(S_a)
        =2\bigl(\rho_a\cos\theta_a+z_a\sin\theta_a\bigr)(S_a)
        =2\,X\cdot T
        \in\{\pm2R\}.
\]
Since $\rho_a^2+z_a^2<R^2$ on $[0,S_a)$, this derivative is nonnegative, hence equal to $2R>0$.  The hitting is therefore automatically transversal, and $a\in\cA$.

\begin{lemma}[Nondegeneracy of the unweighted shooting root]\label{lem:catenoid-shooting-nondegenerate}
Fix $R>0$.  Let $t_*>0$ be the unique solution of
\begin{equation}\label{eq:tstar-critical-catenoid}
        t_*\tanh t_*=1,
\end{equation}
and set
\begin{equation}\label{eq:a0-critical-catenoid}
        a_0=\frac{R}{\sqrt{\cosh^2t_*+t_*^2}}.
\end{equation}
For $c=0$, the shooting map has a root at $a_0$, corresponding to the
critical catenoid, and this root is nondegenerate:
\begin{equation}\label{eq:shooting-nondegenerate}
        \partial_a\widehat{\cQ}(0,a_0)>0.
\end{equation}
\end{lemma}

\begin{proof}
For $c=0$, the solution with neck radius $a$ is the catenary
\begin{equation}\label{eq:scaled-catenary-profile}
        \rho_{0,a}=a\cosh t,
        \qquad
        z_{0,a}=at,
        \qquad
        s=a\sinh t,
\end{equation}
and hence
\begin{equation}\label{eq:catenary-angle}
        \cos\theta_{0,a}=\tanh t,
        \qquad
        \sin\theta_{0,a}=\operatorname{sech}t.
\end{equation}
Write
\[
        D(t)=\cosh^2t+t^2.
\]
Near the critical root, the first hitting parameter $t=t(a)$ is determined
by
\begin{equation}\label{eq:unweighted-hitting-t}
        a^2D(t(a))=R^2.
\end{equation}
At that hitting point the shooting residual is
\begin{equation}\label{eq:unweighted-shooting-explicit}
        \widehat{\cQ}(0,a)
        =a\bigl(1-t(a)\tanh t(a)\bigr).
\end{equation}
Thus \eqref{eq:tstar-critical-catenoid} and
\eqref{eq:a0-critical-catenoid} give a root.  The function
$t\mapsto t\tanh t$ is strictly increasing from $0$ to $+\infty$, so
$t_*$, and hence $a_0$, is unique in the unweighted shooting family.

Differentiating \eqref{eq:unweighted-hitting-t} gives
\[
        t'(a)=-\frac{2D(t(a))}{aD'(t(a))}.
\]
Since the prefactor in parentheses in
\eqref{eq:unweighted-shooting-explicit} vanishes at $a_0$, we obtain
\begin{equation}\label{eq:unweighted-shooting-derivative}
\begin{aligned}
        \partial_a\widehat{\cQ}(0,a_0)
        &=-a_0\bigl(\tanh t_*+t_*\operatorname{sech}^2t_*\bigr)t'(a_0)\\
        &=\frac{2D(t_*)
        \bigl(\tanh t_*+t_*\operatorname{sech}^2t_*\bigr)}{D'(t_*)}>0.
\end{aligned}
\end{equation}
This proves the claimed nondegeneracy.
\end{proof}

\begin{proof}[Proof of Theorem~\ref{thm:local-gaussian-family}]
At $(c,a)=(0,a_0)$, the critical catenoid reaches
$\partial\B_R^3$ transversally.  The right-hand side of
\eqref{eq:rot-ode} is smooth as long as $\rho>0$.  Smooth dependence of
ODE solutions on parameters, followed by the implicit function theorem
applied to
\[
        (s,c,a)\longmapsto
        \rho_{c,a}(s)^2+z_{c,a}(s)^2-R^2,
\]
therefore gives a smooth zero $S(c,a)$ of the hitting equation in a
neighborhood of the critical endpoint.  This zero is still the first hitting
time: on the critical catenoid,
$s\mapsto \rho_{0,a_0}(s)^2+z_{0,a_0}(s)^2$ is strictly increasing for
$s>0$.  Hence it is uniformly smaller than $R^2$ on every compact
subinterval bounded away from the endpoint, and smooth ODE dependence
preserves this strict gap for nearby $(c,a)$.  Together with transversality
near the endpoint, this excludes an earlier hit.  In this neighborhood define
\begin{equation}\label{eq:two-parameter-shooting-map}
        \widehat{\cQ}(c,a)
        =\rho_{c,a}(S(c,a))\sin\theta_{c,a}(S(c,a))
        -z_{c,a}(S(c,a))\cos\theta_{c,a}(S(c,a)).
\end{equation}
Lemma~\ref{lem:catenoid-shooting-nondegenerate} and the implicit function
theorem yield $\varepsilon_R>0$ and a unique smooth function
\begin{equation}\label{eq:local-neck-family}
        a:[0,\varepsilon_R)\longrightarrow(0,R),
        \qquad
        a(0)=a_0,
        \qquad
        \widehat{\cQ}(c,a(c))=0.
\end{equation}
Here uniqueness is only asserted in the parameter neighborhood on which
the implicit function theorem is applied.  Shrinking $\varepsilon_R$ if
necessary, we have $cR^2<2$ for $0<c<\varepsilon_R$.  Reflecting the
profile across $\{z=0\}$ and rotating it once about the $z$-axis produces
a smooth free boundary $f_c$-minimal annulus $\Sigma_c$.

It remains to prove that this family satisfies the pinching.  On the
positive half-profile write
\begin{equation}\label{eq:PQh-definitions}
\begin{aligned}
        P&=\rho\cos\theta+z\sin\theta=\langle X,T\rangle,\\
        Q&=\rho\sin\theta-z\cos\theta,\\
        h&=\frac{\sin\theta}{\rho}=\kappa_2.
\end{aligned}
\end{equation}
The ODE gives $\kappa_1=cQ-h$, and direct differentiation gives the two
identities
\begin{equation}\label{eq:Pprime-and-Qh}
        P'=1-cQ^2+Qh,
        \qquad
        1-Qh=\frac{P\cos\theta}{\rho}.
\end{equation}
Since $|x^\perp|^2=Q^2$ and
\[
        |A_{x^\perp}|^2
        =Q^2\bigl((cQ-h)^2+h^2\bigr),
\]
the pinching deficit admits the exact factorization
\begin{equation}\label{eq:pinching-deficit-factorization}
\begin{aligned}
        \mathscr E
        &:=1+\bigl(1-c|x^\perp|^2\bigr)^2-|A_{x^\perp}|^2\\
        &=2(1-Qh)(1+Qh-cQ^2)\\
        &=2\,\frac{P\cos\theta}{\rho}\,P'.
\end{aligned}
\end{equation}
Thus no estimate has been lost: it suffices to determine the signs of the
three factors $P$, $\cos\theta$, and $P'$.

For the critical catenoid, formulas
\eqref{eq:scaled-catenary-profile}--\eqref{eq:catenary-angle} give, on
$0<t\le t_*$,
\begin{equation}\label{eq:catenoid-sign-quantities}
\begin{aligned}
        P_0&=a_0\bigl(\sinh t+t\operatorname{sech}t\bigr)>0,\\
        Q_0&=a_0\bigl(1-t\tanh t\bigr)\ge0,\\
        h_0&=\frac{1}{a_0\cosh^2t}>0,\\
        P_0'&=1+Q_0h_0\ge1,
        \qquad \cos\theta_0=\tanh t>0.
\end{aligned}
\end{equation}
The only common zero that must be treated separately is the neck $s=0$,
where $P_0(0)=\cos\theta_0(0)=0$.  For the perturbed family,
\begin{equation}\label{eq:neck-first-derivatives}
\begin{aligned}
        (\cos\theta)'(0)
        &=\frac{1}{a(c)}-ca(c),\\
        P'(0)&=2-ca(c)^2.
\end{aligned}
\end{equation}
Both quantities are positive for all sufficiently small $c\ge0$.
Consequently, $P>0$ and $\cos\theta>0$ immediately to the right of the
neck, uniformly for $c$ in a smaller parameter interval.

Away from the neck, the quantities in
\eqref{eq:catenoid-sign-quantities} have positive lower bounds on compact
subintervals.  Identify the varying profile intervals by
$s=\sigma S(c,a(c))$, $0\le\sigma\le1$.  The $C^1$-smooth dependence of
the ODE solution on $(c,a)$, together with
\eqref{eq:neck-first-derivatives}, then permits a further shrinking of
$\varepsilon_R$ so that
\begin{equation}\label{eq:positive-half-signs}
        \rho>0,
        \qquad P>0,
        \qquad \cos\theta>0,
        \qquad P'>0
        \quad\text{for }0<s\le S(c,a(c)).
\end{equation}
Equation \eqref{eq:pinching-deficit-factorization} now gives
$\mathscr E>0$ on the positive half-profile away from the neck.

At the neck,
\[
        Q(0)=a(c),
        \qquad h(0)=\frac1{a(c)},
\]
so $1-Q(0)h(0)=0$ and hence $\mathscr E(0)=0$.  Under reflection,
\[
        \rho(-s)=\rho(s),
        \qquad z(-s)=-z(s),
        \qquad \theta(-s)=\pi-\theta(s).
\]
Thus $P$ and $\cos\theta$ both change sign, while their product and $P'$
are even.  It follows that $\mathscr E>0$ on the negative half-profile
away from the neck as well.  After rotation, equality in
\eqref{eq:local-family-pinching} therefore holds precisely on the minimum
circle.

Finally,
\[
        \kappa_2(0)=\frac1{a(c)}>0,
\]
so $\Sigma_c$ is non-flat.  Moreover,
$\sin\theta_0=\operatorname{sech}t$ has a positive lower bound on the
unweighted half-profile; smooth dependence preserves $\sin\theta>0$ for
small $c$.  Together with $\cos\theta>0$ for $s>0$, this shows that each
half-profile is embedded.  Moreover, because $z'(s)=\sin\theta>0$, the
positive half-profile lies strictly in the upper half-space $\{z>0\}$ for
$s>0$; its reflected half lies in $\{z<0\}$, and the two meet only at the
neck point $z=0$.  After rotation they therefore meet only along the common
minimum circle, so the full surface has no self-intersections.  The standard
one-fold angular parametrization gives an embedded annulus, completing the
proof.
\end{proof}

\begin{question}[Global uniqueness of the rotational shooting root]\label{q:shooting}
Fix $R>0$ and $c>0$ with $cR^2\le2$.  Is a root of
\[
        \cQ(a)=0,\qquad a\in\cA,
\]
unique in the full shooting domain $\cA$?  In the perturbative regime of
Theorem~\ref{thm:local-gaussian-family}, this asks whether the locally
unique root continued from the critical catenoid is the only root in
$\cA$.
\end{question}

\begin{remark}
The derivative $\cQ'(a)$ can be expressed in terms of the boundary value
of the variational Jacobi field along the profile curve together with the
transversality of the hitting time $S_a$.  Proving that $\cQ'(a)$ has a
fixed sign is a nontrivial Sturm--Liouville non-oscillation problem.  A
separate phase-plane analysis may provide useful context, but it is not by
itself a proof of uniqueness for the free boundary shooting root in the
ball.

Theorem~\ref{thm:local-gaussian-family} already resolves existence and
pinching compatibility in a neighborhood of the unweighted critical
catenoid.  Its nondegeneracy calculation proves only local monotonicity of
the two-parameter shooting map near $(0,a_0)$; it does not imply that
$\cQ'(a)$ has one sign on every component of $\cA$.  This is exactly the
distinction between local implicit-function uniqueness and the global
shooting uniqueness asked for in Question~\ref{q:shooting}.
\end{remark}

\begin{proposition}[Conditional global model uniqueness for embedded rotational Gaussian annuli]\label{prop:conditional-uniqueness}
Fix $c>0$ and $R>0$, and let $\cA$ and $\cQ$ be the shooting domain and the
shooting map of \eqref{eq:shooting-domain} and \eqref{eq:shooting-target}.
Suppose that $\cQ$ has a unique root $a_0\in\cA$.  Then every
\emph{embedded} annulus in the non-disk branch of
Theorem~\ref{thm:mainB} is congruent, as an embedded submanifold and up to an
ambient rotation, to the rotational model with neck radius $a_0$.  Under this
additional uniqueness assumption, this model may be regarded as the Gaussian
critical $f$-catenoid associated with $(c,R)$.
\end{proposition}

\begin{proof}
Let $\Sigma$ be such an embedded annulus.  By
Theorem~\ref{thm:mainB}, it is rotationally symmetric, and by
Lemma~\ref{lem:minimum-neck-normalization} its meridian profile is, after an
ambient rotation and a choice of profile orientation, the solution of
\eqref{eq:rot-ode} with the neck data \eqref{eq:neck-initial} for some
$a\in(0,R)$.  As shown after \eqref{eq:shooting-target}, the free boundary
condition forces $S_a<+\infty$, transversal hitting, and $\cQ(a)=0$; hence
$a\in\cA$ is a root of $\cQ$, and by hypothesis $a=a_0$.  The backward branch
of the profile is the reflection of the forward branch across $\{z=0\}$, so
the whole meridian profile is determined by $a_0$.  Embeddedness excludes a
nontrivial angular multiple cover; hence this profile determines $\Sigma$ as
an embedded submanifold, up to an ambient rotation.
\end{proof}

\begin{remark}
The embeddedness hypothesis in Proposition~\ref{prop:conditional-uniqueness}
is essential for this formulation.  If immersions are allowed, composing the
angular variable of a rotational model with a degree-$m$ covering of
$\Sph^1$, $m\ge2$, preserves the profile, the shooting root, and the pinching
condition, but gives a multiply covered immersion.  Without embeddedness, the
corresponding conclusion is uniqueness of the rotational image, not of the
immersion.
\end{remark}

Theorem~\ref{thm:local-gaussian-family} verifies the existence of at least
one root near the critical-catenoid root, but it does not verify the global
uniqueness hypothesis of Proposition~\ref{prop:conditional-uniqueness}:
additional roots outside the implicit-function neighborhood are not
excluded.

\section{Examples and limiting cases}

\begin{example}[Equatorial disks]
Let $M=\Bcl_R^k\times\{0\}\subset \Bcl_R^N$.  Then $x^\perp=0$, $A=0$, and $\vec H=0$.  Thus $M$ is free boundary Gaussian $f$-minimal for every $c\ge0$, and the pinching condition holds trivially.  The conclusion $M\cong\D^k$ is sharp in the disk case.
\end{example}

\begin{remark}
When $c=0$, the equation becomes the usual minimal equation and the radial pinching condition in Theorem~\ref{thm:mainA} reduces to
\[
        |A_{x^\perp}|^2\le \frac{k}{k-1}.
\]
In dimension two and codimension one this is exactly the Ambrozio--Nunes support-function condition
\[
        |A|^2\langle x,\nu\rangle^2\le2.
\]
The critical catenoid realizes the annular alternative in this unweighted limit: the inequality is saturated on its minimum circle, and Ambrozio--Nunes prove that equality at some point forces the critical catenoid \cite{AmbrozioNunes2021}.  Thus the solid-tube alternative is nonempty at $c=0$.

For the genuinely Gaussian case $c>0$,
Theorem~\ref{thm:local-gaussian-family} proves more: for every fixed $R>0$
there is an $\varepsilon_R>0$ such that the solid-tube alternative is nonempty
for every $0<c<\varepsilon_R$, and the resulting embedded rotational
annulus satisfies \eqref{eq:main-pinching}, with equality exactly along its
minimum circle.  Thus the Gaussian equality alternative is non-vacuous.  The result
is perturbative, however; it does not assert continuation over the entire
range $0<cR^2\le2$, nor does it exclude additional shooting roots far from
the critical-catenoid family.

Finally, since
\[
        |A_{x^\perp}|^2
        =\sum_{i,j}\langle x^\perp,A(e_i,e_j)\rangle^2
        \le |x^\perp|^2|A|^2,
\]
the present condition controls only the radial contraction of the second fundamental form.  Nevertheless, for $c>0$ its right-hand side is
\[
        1+\frac{(1-\zeta)^2}{k-1},\qquad \zeta=c|x^\perp|^2,
\]
and this number is less than $k/(k-1)$ when $0<\zeta<2$.  Thus the Gaussian radial-projection pinching is not a pointwise weakening of the Barbosa--Viana full-norm pinching \cite{BarbosaViana2022}.  The correct comparison is structural rather than monotone: the hypothesis uses fewer components of $A$, but with a drift-dependent threshold.  Only at $c=0$ does the scalar radial condition become an actual weakening of the full-norm condition in codimension at least two; in codimension one the scalar quantities coincide.
\end{remark}

\end{document}